\documentclass[leqno,11pt, a4paper]{amsart} %leqno is the option to put formula numbers on the left side
\usepackage{amssymb}
\usepackage{amsmath}
\usepackage{amsthm}
\usepackage{amscd}
\usepackage{graphicx}
\usepackage[dvips]{color}
\usepackage[all]{xy}
\usepackage[dvipdfmx]{hyperref}
\usepackage{mathtools}
\usepackage{palatino}
\usepackage{comment}
\makeatletter
    
    \@addtoreset{equation}{section}
\makeatother
\theoremstyle{plain} %text of this environment is typesetted in italics
\newtheorem{theorem}{\indent\bf Theorem}[section]

\newtheorem{lemma}[theorem]{\indent\bf Lemma}

\newtheorem{proposition}[theorem]{\indent\bf Proposition}
\newtheorem{claim}[theorem]{\indent\bf Claim}
\newtheorem{question}[theorem]{\indent\bf Question}
\theoremstyle{definition} %text of this environment is typesetted in roman letters

\newcommand{\N}{\mathbb N}
\newcommand{\Z}{\mathbb Z}
\newcommand{\Aut}{\mathrm{Aut}}
\newcommand{\FR}{\mathit{FR}}
\newcommand{\ci}[2]{\cite[#1]{#2}}

\begin{document}

\title[Weak density of orbit equivalence classes]
{Weak density of orbit equivalence classes and free products of infinite abelian groups} %title of paper and the running head option

\author[Takaaki Moriyama]{Takaaki Moriyama} %first author's name and the running head option

%%%%%%%%%%%% Authors' addresses %%%%%%%%%%%%%
\address{% First Author
Graduate School of Mathematical Sciences, The University of Tokyo, Komaba, Tokyo 153-8914, Japan
}
\email{takmori570@gmail.com}
%%%%%%%%%%%%%%%%%%%%%%%%%%%%%%%%%%%%%%%%%

\date{\today}

\maketitle

\begin{abstract}
We show that if a countable group $G$ is the free product of infinite abelian groups, then for every free, probability-measure-preserving (p.m.p.) action of $G$, its orbit equivalence class is weakly dense in the space of p.m.p.\ actions of $G$.
This extends Lewis Bowen's result for free groups.
\end{abstract}

\section{Introduction}

In this paper, we are concerned with the space of probability-measure-preserving (p.m.p.) actions of a countable group, elaborated in Kechris' monograph \cite{ke}. 
Before stating our main result, let us briefly review some known results on the space of p.m.p.\ actions.

Throughout the paper, let $G$ be a countable group, and let $(X, \mu)$ be a non-atomic standard probability space unless otherwise stated.
Let $\Aut(X, \mu)$ be the group of all measure-preserving Borel automorphisms of $(X,\mu)$, where two of them are identified if they agree $\mu$-almost everywhere. 
The \textit{weak topology} on $\Aut(X,\mu)$ is defined as the topology generated by the sets
\[\{\, R \in \Aut(X,\mu) \mid \mu(R(P) \bigtriangleup S(P)) < \varepsilon  \text{ for all } P \in \mathcal P \,\}\]
for $S\in \Aut(X, \mu)$, $\varepsilon >0$, and a finite Borel partition $\mathcal P$ of $X$.
It is known that $\Aut(X,\mu)$ is a Polish group with respect to the weak topology (\ci{\S 10 (A)}{ke}).

We mean by a \textit{p.m.p.}\ action of $G$ on $(X, \mu)$ a homomorphism from $G$ into $\Aut(X, \mu)$. 
Let $A(G, X, \mu)$ denote the set of all p.m.p.\ actions of $G$ on $(X,\mu)$, and let $\FR(G,X,\mu)$ denote the set of all essentially free p.m.p.\ actions of $G$ on $(X,\mu)$.
As being in \cite{ke}, the set $A(G,X,\mu)$ is naturally identified with the subspace of the product space $\prod_G\Aut(X,\mu)$ equipped with the product topology of the weak topology.
For $\alpha \in A(G,X,\mu)$ and $g\in G$, we write $g^\alpha = \alpha(g)$. 
We say that two actions $\alpha,\beta \in A(G,X,\mu)$ are \textit{measure-conjugate} if there exists $R \in \Aut(X,\mu)$ such that $g^\alpha x = Rg^\beta R^{-1}x$ for all $g \in G$ and $\mu$-almost every $x \in X$. 
We then write $\alpha = R\beta R^{-1}$. 
For $\alpha \in A(G,X,\mu)$, let $[\alpha]_\mathit{MC}$ denote the set of all actions $\beta \in A(G,X,\mu)$ that are measure-conjugate to $\alpha$.

The space $A(G,X,\mu)$ reflects many analytic properties of the group $G$.
For example, for every infinite amenable group $G$ and for every $\alpha \in \FR(G,X,\mu)$, the set $[\alpha]_\mathit{MC}$ is weakly dense in $A(G,X,\mu)$ (\ci{\S 3.1}{ag}; see also \ci{Remark, p.91}{ke}). 
By contrast, every non-amenable group $G$ has an uncountable antichain in $\FR(G,X,\mu)$ with respect to the pre-order of weak containment (\ci{Remark 4.3}{td}). 
We say that $\alpha \in A(G,X,\mu)$ is \textit{weakly contained} in $\beta \in A(G,Y,\nu)$ (denoted by $\alpha \prec \beta$) if for any Borel subsets $A_1,\ldots,A_n \subset X$ and any finite subset $F \subset G$ and any $\varepsilon > 0$, there exist Borel subsets $B_1,\ldots,B_n \subset Y$ such that for any $g \in F$ and $i,j \in \{1,\ldots,n\}$, we have
\[
|\mu(g^\alpha(A_i) \cap A_j) - \nu(g^\beta(B_i) \cap B_j)| < \varepsilon.
\]
For two actions $\alpha, \beta \in A(G,X,\mu)$, it is known that $\alpha$ is weakly contained in $\beta$ if and only if $\alpha$ belongs to the weak closure of $[\beta]_\mathit{MC}$ (\ci{Theorem 2.3}{bk}).
For more details on the weak topology on $A(G,X,\mu)$ and weak containment, we refer the reader to \cite{bk} and \cite{ke}.

We say that two actions $\alpha, \beta \in A(G,X,\mu)$ are \textit{orbit equivalent} if there exists $R \in \Aut(X,\mu)$ such that $G^\alpha x = RG^\beta R^{-1} x$ for $\mu$-almost every $x \in X$. 
By definition, two measure-conjugate actions in $A(G,X,\mu)$ are orbit equivalent.
For $\alpha \in A(G,X,\mu)$, let $[\alpha]_\mathit{OE}$ denote the set of all actions $\beta \in A(G,X,\mu)$ that are orbit equivalent to $\alpha$. 

Lewis Bowen proved the following:

\begin{theorem}[\ci{Theorem 1.1}{bo}]\label{thm-bowen}
Let $G$ be a free group with at most countably many generators.
Then for any $\alpha \in \FR(G,X,\mu)$, the set $[\alpha]_\mathit{OE}$ is weakly dense in $A(G, X, \mu)$.
\end{theorem}

Some applications of this result are given in \cite[Remarks 1 and 2]{bo}. 
As in \cite[Theorem 1.2]{bo}, it turns out from orbit equivalence rigidity that some non-amenable groups do not satisfy the conclusion of Theorem \ref{thm-bowen}. 
In addition to this, any countable group with property (T) does not satisfy the conclusion of Theorem \ref{thm-bowen}.
In fact, if a countable group $G$ has property (T), then the set of all ergodic p.m.p.\ actions of $G$ on $(X,\mu)$ is closed in $A(G,X,\mu)$ and has no interior (see \ci{Theorem 12.2, i)}{ke}, the proof of which is based on \ci{Theorem 1}{gw}).  
It follows from the result of \ci{\S 3.1}{ag} mentioned above that all infinite amenable groups satisfy the conclusion of Theorem \ref{thm-bowen}.
Bowen asked the following:

\begin{question}[\ci{Question 1}{bo}]\label{question}
Which countable groups $G$ satisfy the conclusion of Theorem \ref{thm-bowen}?
For example, do all strongly treeable groups (e.g., $\textit{PSL}(2, \Z)$) satisfy this conclusion?
\end{question}

The goal of this paper is to present a new class of examples of such groups in the following:

\begin{theorem}\label{thm-intro}
Let $G$ be the free product of at most countably many, countably infinite abelian groups.
Then for any $\alpha \in \FR(G,X,\mu)$, the set $[\alpha]_\mathit{OE}$ is weakly dense in $A(G,X,\mu)$.
\end{theorem}

Toward the proof of Theorem \ref{thm-intro}, in Section 2, we generalize Bowen's lemma of good partition \ci{Lemma 4.2}{bo}. In Section 3, we prove Theorem \ref{thm-intro}. In contrast with Bowen's proof of Theorem \ref{thm-bowen}, our proof of Theorem \ref{thm-intro} depends on Foreman-Weiss' argument in the proof of \ci{Theorem 16}{fw} and the Rohlin lemma for tiles due to Ornstein-Weiss \ci{II. \S 2, Theorem 5}{ow}. For our purpose, we slightly strengthen the conclusion of the Rohlin lemma for tiles, assuming that the acting group is abelian (Lemma \ref{Rohlin}).
We note that our proof of Theorem \ref{thm-intro} cannot be applied to the case when $G$ is the free product of finite groups, for example, $\textit{PSL}(2,\Z)$.

\medskip

\textbf{Acknowledgments.} 
The author would like to thank his supervisor, Professor Yoshikata Kida, for his helpful comments and enormous support.
The author would also like to thank the anonymous referee for the careful reading of the paper and many helpful suggestions.
This work was supported by the Program for Leading Graduate Schools, MEXT, Japan.

%%%%%%%%%%%%%%%%%%%%%%%%%%%%%%%%%%%%%%%%%%%%

\section{A lemma of good partition}

The aim of this section is to show Lemma \ref{goodpart}, which extends \ci{Lemma 4.2}{bo} to that for the free products of infinite amenable groups. 
Let $\alpha \in A(G,X,\mu)$.
For a finite subset $F \subset G$ and a function $f$ on $X$, we define the averaging function $A_{F}^\alpha[f]$ on $X$ by
\[
A_{F}^\alpha[f](x) \coloneqq \frac{1}{|F|} \sum_{g \in F} f(g^{\alpha}x)
\]
for $x\in X$.
Let $\theta \colon (X, \mu)\to (W, \omega)$ be the ergodic decomposition map for the action $\alpha$ with disintegration $\mu =\int_W \mu_w\, d\omega(w)$.
We then set $\mu_x=\mu_{\theta(x)}$ for $x\in X$ and also call the disintegration $\mu =\int_X\mu_x\, d\mu(x)$ the ergodic decomposition of $\mu$ with respect to the action $\alpha$ in the sequel if there is no cause of confusion.

\begin{lemma} \label{goodpart}
Let $G_1,\ldots, G_k$ be countably infinite amenable groups and define $G$ as the free product $G = G_1 \ast \cdots \ast G_k$. Let $\alpha \in \FR(G,X,\mu)$ and let $\alpha|G_i$ denote the restriction of $\alpha$ to $G_i$ for each $i\in \{ 1,\ldots,k\}$.
Then for any probability measure $\pi$ on a finite set $\mathcal A$ and any $\varepsilon >0$, there exists a Borel map $\psi \colon X\to \mathcal A$ such that $\psi_\ast \mu = \pi$ and the following holds:
For each $i\in \{ 1,\ldots, k\}$, if $\mu = \int_{X} \mu_x^i \,d\mu(x)$ is the ergodic decomposition of $\mu$ with respect to $\alpha|G_i$, then
\[
\mu(\{\, x \in X \mid \max_{a \in \mathcal A} |(\psi_*\mu_x^i) (a)-\pi(a)| >2\varepsilon \,\}) <\varepsilon.
\]
\end{lemma}

\begin{proof}
The proof basically follows that of \ci{Lemma 4.2}{bo}.
We may assume that $\pi$ is not a measure supported on a single point of $\mathcal{A}$.
Fix $i\in \{ 1,\ldots, k\}$ and pick a F\o lner sequence $\{F_n^i \}_{n\in \N}$ for $G_i$. Let $s_i$ denote the Bernoulli action $G_i\curvearrowright (\mathcal{A}^{G_i}, \pi^{G_i})$, which is defined by $(g^{s_i}y)(h) =y(hg)$ for $g,h \in G$ and $y \in {\mathcal A}^{G_i}$. We set $C_a^i=\{\, y \in \mathcal A^{G_i} \mid y(e) = a \,\}$ for $a \in \mathcal A$. Then we have
\[
A_{F_n^i}^{s_i}[1_{C_a^i}](y)=\frac{1}{|F_n^i|}\sum_{g\in F_n^i}1_{C_a^i}(g^{s_i}y) = \frac{|\{\, g \in F_n^i \mid y(g) = a \,\}|}{|F_n^i|}
\]
for any $y\in A^{G^i}$. If there is no cause of confusion, we often regard $A_{F_n^i}^{s_i}[1_{C_a^i}]$ as the function on $\mathcal A^{F_n^i}$ defined by the right-hand side of the above equation. According to the $L^1$ version of the mean ergodic theorem (for example, see \ci{Theorem 4.23}{kl}), we have
\[
\lim_{n\to \infty} \| A_{F_n^i}^{s_i}[1_{C_a^i}] - \pi(a) \|_{L^1(\pi^{G_i})} = 0
\]
for every $a \in \mathcal A$. This implies that there exists $n_i \in \N$ such that if $n\geq n_i$, then 
\[
\pi^{G_i} (\{\, y\in \mathcal A^{G_i} \mid \max_{a \in \mathcal A} |A_{F_n^i}^{s_i}[1_{C_a^i}](y) - \pi(a)| >\varepsilon \,\}) < \frac{\varepsilon^2}{2}
\]
and therefore
\begin{equation}\label{pi-fni}
\pi^{F_n^i} (\{\, y\in \mathcal A^{F_n^i} \mid \max_{a \in \mathcal A} |A_{F_n^i}^{s_i}[1_{C_a^i}](y) - \pi(a)| >\varepsilon \,\}) < \frac{\varepsilon^2}{2}.
\end{equation}

We set $n_0=\max_{i\in \{ 1,\ldots,k\}}n_i$.
We set $F^i = F_{n_0}^i$ for the ease of symbols and set $F= \bigcup_{i=1}^{k} F^i$. 
We claim that there exists a Borel map $\psi \colon X\to \mathcal A$ such that $\psi_\ast \mu = \pi$ and for each $i\in \{ 1,\ldots,k\}$, if we define a map $\Psi_i \colon  X\to \mathcal A^{F^i}$ by $\Psi_i(x)(g) = \psi(g^{\alpha}x)$ for $x\in X$ and $g \in F^i$, then
\begin{equation}\label{psi_i}
\max_{y \in \mathcal A^{F^i}} |((\Psi_i)_\ast \mu)(y) - \pi^{F^i}(y)|  < \frac{\varepsilon^2}{2M},
\end{equation}
where we set $M = {|\mathcal A|}^{|F|}$.
Indeed, by \ci{Theorem 1}{aw}, the Bernoulli action $G\curvearrowright (\mathcal{A}^G,\pi^G)$, denoted by $s$, is weakly contained in any free p.m.p.\ action of $G$ on $(\mathcal{A}^G,\pi^G)$.
Since both $(X,\mu)$ and $(\mathcal A^G,\pi^G)$ are non-atomic standard probability spaces, there exists a measure-space isomorphism $R \colon (X,\mu) \to (\mathcal A^G,\pi^G)$. Then we see that $s$ belongs to the weak closure of $[R \alpha R^{-1}]_\mathit{MC}$ by \ci{Theorem 2.3}{bk}.
This implies that there exists an automorphism $S \in \Aut(\mathcal A^G, \pi^G)$ such that if $C_a=\{\, y \in \mathcal A^G \mid y(e) = a \,\}$ denotes the cylindrical subset for $a \in \mathcal A$, then
\[
\Biggl| \pi^G \Biggr( \bigcap_{g \in F} SR(g^{-1})^{\alpha}R^{-1}S^{-1} C_{a_g}  \Biggr)  - \pi^G \Biggr( \bigcap_{g \in F} (g^{-1})^s C_{a_g} \Biggr)  \Biggr| < \frac{\varepsilon^2}{2M}
\]
for any $(a_g)_{g \in F} \in \mathcal A^F$. Let $\psi \colon  (X,\mu) \to (\mathcal A, \pi)$ be the measure-preserving map defined by the condition $\psi^{-1}(a)=R^{-1}S^{-1}(C_a)$ for each $a \in \mathcal A$. Then
\[
\Biggl| \mu \Biggr( \bigcap_{g \in F} (g^{-1})^{\alpha} \psi^{-1}(a_g) \Biggr) - \pi^{F} ((a_g)_{g \in F}) \Biggr| < \frac{\varepsilon^2}{2M}
\]
for any $(a_g)_{g \in F} \in \mathcal A^F$, which implies our claim.

We prove that the map $\psi$ is a desired one. Fix  $i \in \{1,\ldots,k\}$. We set
\[Z_i = \{\, x\in X \mid \max_{a\in \mathcal A}|A_{F^i}^{s_i}[1_{C_a}](\Psi_i(x)) - \pi(a)| >\varepsilon \,\}.\]
By inequalities (\ref{pi-fni}) and (\ref{psi_i}), we have
\begin{align*}
\mu(Z_i) &= ((\Psi_i)_\ast \mu)(\{\, y\in \mathcal A^{F^i} \mid \max_{a\in \mathcal A}|A_{F^i}^{s_i}[1_{C_a}](y) - \pi(a)| >\varepsilon \,\}) \\
&< {|\mathcal A|}^{|F^i|}  \max_{y \in \mathcal A^{F^i} } | ((\Psi_i)_\ast \mu)(y) - \pi^{F^i}(y) | + \frac{\varepsilon^2}{2} < \varepsilon^2.
\end{align*}
Let $\mu = \int_{X} \mu_x^i \,d\mu(x)$ be the ergodic decomposition of $\mu$ with respect to the action $\alpha|G_i$.
Then $\int_{X} \mu_x^i(Z_i)\,d\mu(x) = \mu(Z_i) <\varepsilon^2$ and hence
\[
\mu (\{\, x \in X \mid \mu_x^i(Z_i)>\varepsilon \,\}) < \varepsilon.
\]
Suppose that a point $x \in X$ satisfies $\mu_{x}^i(Z_i)\leq \varepsilon$. 
Note that for any $a\in \mathcal{A}$ and any $x'\in X$, the following equation holds:
\begin{align*}
A_{F^i}^{s_i}[1_{C_a}](\Psi_i(x'))= \frac{|\{\, g \in F^i \mid \psi(g^\alpha x') = a \,\}|}{|F^i|} =\frac{1}{|F^i|}\sum_{g\in F^i}1_{\psi^{-1}(a)}(g^\alpha x').
\end{align*}
Therefore for any $a\in \mathcal{A}$, we have $\int_X A_{F^i}^{s_i}[1_{C_a}](\Psi_i(x'))\,d\mu_{x}^i(x') = (\psi_\ast \mu_{x}^i)(a)$ and
\begin{align*}
&|(\psi_\ast \mu_{x}^i)(a) - \pi(a) | = \left| \int_X ( A_{F^i}^{s_i}[1_{C_a}](\Psi_i(x')) - \pi(a) ) \, d\mu_{x}^i(x') \right| \\
&\leq \int_{Z_i} | A_{F^i}^{s_i}[1_{C_a}](\Psi_i(x')) - \pi(a) | \, d\mu_{x}^i(x') + \int_{X\setminus Z_i} | A_{F^i}^{s_i}[1_{C_a}](\Psi_i(x')) - \pi(a) | \, d\mu_{x}^i(x') \\
&\leq \mu_{x}^i({Z_i}) + \varepsilon \leq 2\varepsilon.
\end{align*}
Hence
\[
\mu(\{\, x \in X \mid \max_{a \in \mathcal A} |(\psi_*\mu_{x}^i)(a)-\pi(a)| >2\varepsilon \,\}) \leq \mu (\{\, x \in X \mid \mu_{x}^i(Z_i)>\varepsilon \,\}) < \varepsilon.\qedhere
\]
\end{proof}

%%%%%%%%%%%%%%%%%%%%%%%%%%%%%%%%%%%%%%%%%

\section{Proof of the Main Theorem}
A key ingredient of the proof of Theorem \ref{thm-intro} is the Rohlin lemma for tiles.
For a countable group $G$, we say that a finite subset $T$ of $G$ is a {\it tile} for $G$ if there exists a subset $C \subset G$ such that the family $\{Tc\}_{c \in C}$ is pairwise disjoint and $G = \bigcup_{c\in C} Tc$. First let us recall the following fact due to Ornstein-Weiss:

\begin{lemma}[\ci{II.\S 2, Theorem 5}{ow}; see also \ci{Theorem 3.3}{o1}]\label{Rohlintile}
Let $G$ be a countable amenable group and $T$ a tile for $G$. Let $G \curvearrowright (X,\mu)$ be an essentially free p.m.p.\ action.
Then for any $\varepsilon >0$, there exists a Borel subset $B \subset X$ such that the family $\{tB\}_{t\in T}$ is pairwise disjoint and $\mu(\bigcup_{t \in T} tB) > 1-\varepsilon$.
\end{lemma}

Assuming that $G$ is abelian, we strengthen Lemma \ref{Rohlintile} into the following, which will be important in the proof of Theorem \ref{thm-intro}:

\begin{lemma}\label{Rohlin}
Let $G$ be a countable abelian group and $T$ a tile for $G$. Let $G \curvearrowright (X,\mu)$ be an essentially free p.m.p.\ action.
Then for any $\varepsilon >0$ and any Borel subset $A \subset X$ with $\mu(A)<\varepsilon/2$, there exists a Borel subset $B \subset X$ such that the family $\{tB\}_{t\in T}$ is pairwise disjoint, $\mu(\bigcup_{t \in T} tB) > 1-\varepsilon$, and $A \cap B = \emptyset$.
\end{lemma}

\begin{proof}
Since $G$ is amenable and $T$ is a tile, it follows from Lemma \ref{Rohlintile} that there exists a Borel subset $W \subset X$ such that the family $\{tW\}_{t\in T}$ is pairwise disjoint and $\mu(\bigcup_{t\in T} tW) > 1-\varepsilon/2$.
We choose an element $t_0 \in T$ such that $\mu(t_0W \cap A) \leq \mu(tW \cap A)$ for any $t \in T$. Then $\mu(t_0W \cap A) < \varepsilon/(2|T|)$ because otherwise the inequality $\mu(A) \geq \sum_{t \in T} \mu(tW \cap A) \geq \varepsilon/2$ would hold, which contradicts our assumption.

We set $B = (t_0 W) \setminus A$. Then the family $\{tB\}_{t \in T}$ is pairwise disjoint since $G$ is abelian and the family $\{tW\}_{t \in T}$ is pairwise disjoint. Moreover, we have $\mu(B) > \mu(W) - \varepsilon/(2|T|)$ and hence $\mu( \bigcup_{t\in T} tB ) = |T|\mu(B) > |T|\mu(W) - \varepsilon /2 > 1- \varepsilon$.
Therefore the set $B$ is a desired one.
\end{proof}

For finite subsets $F,T \subset G$ and $\varepsilon >0$, we say that $T$ is {\it $(F,\varepsilon)$-invariant} if $|T \bigtriangleup gT| < \varepsilon |T|$ for every $g \in F$.
We say that a countable amenable group $G$ is \textit{monotilable} if there exists a F\o lner sequence $\{ F_n \}_{n \in \N}$ for $G$ such that each $F_n$ is a tile for $G$. The following fact is briefly mentioned in \ci{I.\S 2, p.22}{ow} and the proof of \ci{Theorem 2}{o2}. We give its proof for the reader's convenience.

\begin{proposition}\label{abtile}
Every countable abelian group is monotilable.
\end{proposition}

\begin{proof}
Let $G$ be a countable abelian group.
If $G$ is finitely generated, then $G$ is isomorphic to the group $\Z^r \oplus C$ for some non-negative integer $r$ and some finite abelian group $C$. In this case, if we set $F_n = ([-n,n]^r\cap \Z^r) \oplus C$, then $F_n$ is a tile for $G$ and $\{ F_n \}_{n \in \N}$ is a F\o lner sequence for $G$, which proves the proposition.
Suppose that $G$ is not finitely generated. Since $G$ is countable, there exists an increasing sequence $\{ G_m \}_{m\in \N}$ of finitely generated subgroups of $G$ such that $G = \bigcup_{m \in \N} G_m$. Considering the right coset decomposition of $G$, we see that every tile for $G_m$ is also a tile for $G$. Then for any finite subset $F \subset G$ and any $\varepsilon >0$, choosing $m \in \N$ such that $F \subset G_m$, we have an $(F,\varepsilon)$-invariant tile $T$ for $G_m$, which is also an $(F, \varepsilon)$-invariant tile for $G$. Thus the proposition follows.
\end{proof}

For the proof of the main result, we also need the following:

\begin{lemma}\label{ergdense}
Let $G_1,\ldots, G_k$ be countably infinite amenable groups and define $G$ as the free product $G = G_1 \ast \cdots \ast G_k$. Then the set of all actions $\alpha \in A(G, X, \mu)$ such that for each $i\in \{ 1, \ldots, k\}$, the restriction $\alpha|G_i$ is essentially free and ergodic is weakly dense in $A(G,X,\mu)$. 
\end{lemma}

\begin{proof}
Let $\beta \in A(G,X,\mu)$. Let $F \subset G$ be a finite subset, let $\varepsilon >0$, and let $A_1, \ldots, A_n \subset X$ be Borel subsets. 
We choose an $N \in \N$ and a finite subset $F_i \subset G_i$ for each $i\in \{ 1,\ldots,k\}$ such that $F \subset (F_1 \cup  \cdots \cup F_k)^N$.
We set $L= (F_1 \cup  \cdots \cup F_k)^N$.
Since $G_i$ is amenable, the set of all essentially free ergodic p.m.p.\ actions of $G_i$ is weakly dense in $A(G_i,X,\mu)$ (\ci{Proposition 13.2}{ke}). Hence, for each $i\in \{ 1,\ldots,k\}$, we can choose an essentially free ergodic p.m.p.\ action $\alpha_i$ of $G_i$ on $(X, \mu)$ such that for any $h \in F_i$, $l \in L$ and $j\in \{ 1,\ldots,n\}$, we have
\[
\mu( h^{\alpha_i} (l^{\beta}A_j) \bigtriangleup h^{\beta} (l^{\beta}A_j)) < \frac{\varepsilon}{N}.
\]
Let $\alpha$ be the action of $G$ on $(X, \mu)$ defined by $\alpha|G_i = \alpha_i$ for each $i\in \{ 1,\ldots,k\}$.
Using the inequality
\[
\mu((g_2g_1)^\alpha A \bigtriangleup (g_2g_1)^\beta A) \leq \mu(g_1^\alpha A \bigtriangleup g_1^\beta A) + \mu(g_2^\alpha (g_1^\beta A) \bigtriangleup g_2^\beta (g_1^\beta A))
\]
for $g_1,g_2 \in G$ and a Borel subset $A \subset X$, we have $\mu(g^{\alpha}A_j \bigtriangleup g^{\beta} A_j) < \varepsilon$ for any $g \in F$ and $j\in \{ 1,\ldots,n\}$, which completes the proof of the lemma.
\end{proof}

We now prove our main result for the free product of finitely many, countably infinite abelian groups:

\begin{theorem}\label{main}
Let $G_1,\ldots, G_k$ be countably infinite abelian groups and define $G$ as the free product $G = G_1 \ast \cdots \ast G_k$. Suppose that $\alpha \in \FR(G,X,\mu)$ and $\beta \in A(G,X,\mu)$. Let $F \subset G$ be a finite subset, let $\varepsilon >0$, and let $A_1,\ldots,A_n \subset X$ be Borel subsets.
Then there exists $\gamma \in [\alpha]_\mathit{OE}$ such that for any $j\in \{ 1,\ldots,n\}$ and $g\in F$, we have $\mu(g^{\gamma} A_j \bigtriangleup g^\beta A_j) < \varepsilon$.
Therefore for any $\alpha \in \FR(G,X,\mu)$, its orbit equivalence class $[\alpha]_{OE}$ is weakly dense in $A(G,X,\mu)$.
\end{theorem}

\begin{proof}
By Lemma \ref{ergdense}, we may assume that for each $i\in \{ 1, \ldots, k\}$, the restriction $\beta|G_i$ is essentially free and ergodic. Moreover, using a similar argument as in the proof of Lemma \ref{ergdense}, we may assume that $F$ is the union $F_1 \cup  \cdots \cup F_k$, where $F_i$ is a finite subset of $G_i$.

We take the finite set $\mathcal A$ and the Borel map $\phi \colon X \to \mathcal A$ such that $\{ \phi^{-1}(a)\}_{a \in \mathcal A}$ is the Borel partition of $X$ generated by the family $\{ g^{\beta}A_j \}_{g\in F, 1\leq j\leq n}$. Let $\varepsilon' = \varepsilon/(24|\mathcal A|)$. By applying Lemma \ref{goodpart} to the action $\alpha$, the probability measure $\phi_\ast \mu$ on $\mathcal A$, and the number $\varepsilon'$, we obtain a Borel map $\psi \colon X\to \mathcal A$ such that $\psi_\ast \mu = \phi_\ast \mu$ and for each $i\in \{ 1,\ldots,k\}$, if $\mu = \int_{X} \mu_x^i \,d\mu(x)$ is the ergodic decomposition of $\mu$ with respect to $\alpha|G_i$, we have
\[
\mu(\{\, x\in X \mid \max_{a \in \mathcal A} |(\psi_\ast\mu_x^i)(a)-(\psi_\ast\mu)(a) | >2\varepsilon' \,\}) <\varepsilon'.
\]
Since $\psi_\ast \mu = \phi_\ast \mu$ and $(X,\mu)$ is a non-atomic standard probability space, there exists $R \in \Aut(X,\mu)$ such that $R(\psi^{-1}(a)) = \phi^{-1}(a)$ for all $a\in \mathcal A$. We set $\alpha' = R\alpha R^{-1} \in \FR(G,X,\mu)$ and set $\alpha_i = \alpha|G_i$, $\alpha'_i = \alpha'|G_i$ and $\beta_i = \beta|G_i$.
We fix $i \in \{1,\ldots,k\}$ throughout Claims \ref{erg+Rohlin}-- \ref{weakapp}.

\begin{claim}\label{erg+Rohlin}
There exist Borel subsets $B^{\alpha'_i}, B^{\beta_i} \subset X$ and an $(F_i,\varepsilon')$-invariant tile $T$ for $G_i$ such that $\mu(B^{\alpha'_i})=\mu(B^{\beta_i})$, $e \in T$, $|T| > 1/\varepsilon'$, and the following conditions (1)--(3) and (1')--(3') hold:

\begin{enumerate}
\item The family $\{ t^{\alpha'_i} B^{\alpha'_i} \}_{t\in T}$ is pairwise disjoint.
\item We have $\mu ( \bigcup_{t\in T} t^{\alpha'_i}B^{\alpha'_i} ) > 1-8\varepsilon'$. 
\item For all $x \in B^{\alpha'_i}$ and $a \in \mathcal A$, we have
\[
\left| \displaystyle\frac{ |\{\, t\in T \mid t^{\alpha'_i}x \in \phi^{-1}(a) \,\}| }{|T|} - \mu(\phi^{-1}(a)) \right| \leq 3\varepsilon'.
\]
\end{enumerate} 
\begin{enumerate}
\renewcommand{\labelenumi}{(\arabic{enumi}')}
\item The family $\{ t^{\beta_i} B^{\beta_i} \}_{t\in T}$ is pairwise disjoint.
\item We have $\mu ( \bigcup_{t\in T} t^{\beta_i}B^{\beta_i} ) > 1-8\varepsilon'$. 
\item For all $x \in B^{\beta_i}$ and $a \in \mathcal A$, we have
\[
\left| \displaystyle\frac{ |\{\, t\in T \mid t^{\beta_i}x \in \phi^{-1}(a) \,\}| }{|T|} - \mu(\phi^{-1}(a)) \right| \leq 3\varepsilon'.
\]
\end{enumerate}  
\end{claim}

\begin{proof}
By the proof of Proposition \ref{abtile}, there exists a F\o lner sequence $\{D_n\}_{n\in \N}$ for $G_i$ such that each $D_n$ is a tile for $G_i$ and contains the identity $e$.
Note that for any $\delta \in A(G_i, X, \mu)$ and any Borel subset $C\subset X$, we have
\[
A_{D_n}^\delta[1_{C}](x) = \displaystyle\frac{ |\{\, g\in D_n \mid g^{\delta}x \in C \,\}| }{|D_n|}
\]
for any $x\in X$.
By the $L^1$ version of the mean ergodic theorem, it follows that for each $a\in \mathcal A$, we have
\[
\displaystyle\lim_{n\to \infty} \| A_{D_n}^{\alpha_i}[1_{\psi^{-1}(a)}] - \mathbb E(1_{\psi^{-1}(a)}) \|_{L^1(\mu)} = 0,
\]
where $\mathbb E(1_{\psi^{-1}(a)})$ is the conditional expectation of $1_{\psi^{-1}(a)}$ onto the space of $\alpha_i$-invariant functions in $L^1(X,\mu)$. Then there exists $n_1 \in \N$ such that for every $n \geq n_1$, we have
\[
 \mu(\{\, x\in X \mid \ \max_{a\in \mathcal A}| A_{D_n}^{\alpha_i}[1_{\psi^{-1}(a)}](x) - \mathbb{E}(1_{\psi^{-1}(a)})(x) | > \varepsilon' \,\} ) < \varepsilon'.
\]
Recall that $\mu =\int_X\mu_x^i\, d\mu(x)$ denotes the ergodic decomposition of $\mu$ with respect to $\alpha_i$. Then $\mathbb{E}(1_{\psi^{-1}(a)})(x) = \mu_x^i(\psi^{-1}(a))$ for any $a\in \mathcal{A}$ and $\mu$-almost every $x \in X$. We set
\[
X_n^{\alpha_i} = \{\, x\in X \mid \ \max_{a \in \mathcal A}| A_{D_n}^{\alpha_i}[1_{\psi^{-1}(a)}](x) - \mu_x^i(\psi^{-1}(a)) | \leq \varepsilon' \  \text{and} \ \max_{a\in \mathcal A} |(\psi_\ast \mu_x^i)(a)-(\psi_\ast\mu)(a)|  \leq 2\varepsilon'  \,\}
\]
for $n \in \N$. Then $\mu(X_n^{\alpha_i}) > 1-2\varepsilon'$ for every $n \geq n_1$. 
If $x \in X_n^{\alpha_i}$, then for each $a \in \mathcal A$, we have
\begin{align*}
& \left| A_{D_n}^{\alpha_i}[1_{\psi^{-1}(a)}](x) - \mu(\psi^{-1}(a)) \right| \\
&\leq \left| A_{D_n}^{\alpha_i}[1_{\psi^{-1}(a)}](x) - \mu_x^i (\psi^{-1}(a)) \right| + |\mu_x^i (\psi^{-1}(a)) - \mu(\psi^{-1}(a))| \leq \varepsilon' + 2\varepsilon' = 3\varepsilon'.
\end{align*}
We set $X_n^{\alpha'_i} = RX_n^{\alpha_i}$. For each $g \in G_i$ and $x\in X$, by the definition of $R$ and $\alpha'$, we have $g^{\alpha_i}x \in \psi^{-1}(a)$ if and only if $g^{\alpha'_i}Rx \in \phi^{-1}(a)$. Combining this with $\psi_\ast \mu = \phi_\ast \mu$, for all $x\in X_n^{\alpha'_i}$ and $a\in \mathcal A$, we have
\[
\left| A_{D_n}^{\alpha'_i}[1_{\phi^{-1}(a)}](x) - \mu(\phi^{-1}(a)) \right| \leq  3\varepsilon'.
\]
On the other hand, since $\beta_i$ is ergodic, there is $n_2 \in \N$ such that for every $n \geq n_2$, we have
\[
\mu( \{\, x \in X \mid | A_{D_n}^{\beta_i} [1_{\phi^{-1}(a)}](x) - \mu(\phi^{-1}(a)) | > 3\varepsilon' \,\} ) < 2\varepsilon'.
\]
We set $X_n^{\beta_i} = \{\, x \in X \mid \max_{a\in \mathcal{A}}| A_{D_n}^{\beta_i} [1_{\phi^{-1}(a)}](x) - \mu(\phi^{-1}(a)) | \leq 3\varepsilon' \,\}$.

We construct Rohlin towers for $\alpha_i'$ and $\beta_i$. Let $m$ be an integer such that $m \geq \max\{n_1,n_2\}$ and $|D_{m}| > 1/\varepsilon'$. We set $A=X \setminus (X_{m}^{\alpha'_i} \cap X_{m}^{\beta_i})$ and set $T=D_{m}$. Since $\mu(X_{m}^{\alpha'_i}) > 1-2\varepsilon'$ and $\mu(X_{m}^{\beta_i}) > 1-2\varepsilon'$, we have $\mu(A) < 4\varepsilon'$. We apply Lemma \ref{Rohlin} to the tile $T$ and the Borel subset $A \subset X$, and then obtain Borel subsets $B^{\alpha_i'}, B^{\beta_i} \subset X$ disjoint from $A$ and satisfying conditions (1), (2), (1') and (2'). Since $B^{\alpha_i'}$ and $B^{\beta_i}$ are disjoint from $A$, conditions (3) and (3') also hold. After replacing one of $B^{\alpha'_i}$ and $B^{\beta_i}$ into its Borel subsets, we may further assume that $\mu(B^{\alpha'_i}) = \mu(B^{\beta_i})$.
\end{proof}

\begin{claim}\label{basepart}
Let $B^{\alpha'_i}, B^{\beta_i}\subset X$ be the Borel subsets and $T$ the tile for $G$ chosen in Claim \ref{erg+Rohlin}. 
Then there exist finite Borel partitions $\{ Q_s^{\alpha'_i} \}_{s=1}^r$ and $\{ Q_s^{\beta_i} \}_{s=1}^r$ of $B^{\alpha'_i}$ and $B^{\beta_i}$, respectively, such that for each $s \in \{ 1,\ldots,r\}$, the following properties hold:
\begin{enumerate}
\renewcommand{\labelenumi}{(\alph{enumi})}
\item The equation $\mu(Q_s^{\alpha'_i}) = \mu(Q_s^{\beta_i})$ holds.
\item For every $t\in T$, we have $t^{\alpha'_i} Q_s^{\alpha'_i} \subset \phi^{-1}(a)$ and $t^{\beta_i} Q_s^{\beta_i} \subset \phi^{-1}(a')$ for some $a,a' \in \mathcal A$.
\item There exist a subset $T_s\subset T$ and a permutation $\sigma_s$ of $T$ such that $|T \setminus T_s| < 7 \varepsilon' |\mathcal A||T|$, $\sigma_s(e)=e$, and for any $t \in T_s$ and any $a \in \mathcal A$, we have 
\begin{equation}\label{equiv1}
\sigma_s(t)^{\alpha'_i}Q_s^{\alpha'_i} \subset \phi^{-1}(a) \Leftrightarrow t^{\beta_i} Q_s^{\beta_i} \subset \phi^{-1}(a).
\end{equation}
\end{enumerate}
\end{claim}

\begin{proof}
We rely on the proof of \cite[Theorem 16]{fw}. 
Let $\{ P_l^{\alpha'_i}\}_l$ be the finite partition of $B^{\alpha'_i}$ generated by $\{ (t^{-1})^{\alpha'_i}(t^{\alpha'_i}B^{\alpha'_i} \cap \phi^{-1}(a))  \}_{t\in T, a \in \mathcal A}$, and $\{ P_m^{\beta_i}\}_m$ be the finite partition of $B^{\beta_i}$ generated by $\{ (t^{-1})^{\beta_i}(t^{\beta_i}B^{\beta_i} \cap \phi^{-1}(a))  \}_{t\in T, a \in \mathcal A}$. Since $\mu(B^{\alpha_i'})=\mu(B^{\beta_i})$, we can take a refinement $\{ Q_s^{\alpha'_i} \}_{s=1}^r$ of $\{ P_l^{\alpha'_i}\}_l$ and a refinement $\{ Q_s^{\beta_i} \}_{s=1}^r$ of $\{ P_m^{\beta_i} \}_m$ such that $\{ Q_s^{\alpha'_i} \}_{s=1}^r$ and $\{ Q_s^{\beta_i} \}_{s=1}^r$ satisfy condition (a). By the construction of $\{ P_l^{\alpha'_i}\}_l$ and $\{ P_m^{\beta_i}\}_m$, they also satisfy condition (b).

We fix $s \in \{1,\ldots,r \}$. We find a subset $T_s$ of $T$ and a permutation $\sigma_s$ of $T$ which satisfy condition (c). By conditions (3) and (3') in Claim \ref{erg+Rohlin}, we have
\[
\left| |\{\, t\in T \mid t^{\alpha'_i}x \in \phi^{-1}(a) \,\} | - | \{\, t\in T \mid t^{\beta_i}x' \in \phi^{-1}(a) \,\} | \right| \leq 6\varepsilon'|T|,
\]
for any $a \in \mathcal A$, $x \in B^{\alpha'_i}$ and $x' \in B^{\beta_i}$.
For each $a\in \mathcal A$, the set $\{\, t\in T \mid t^{\alpha'_i}x \in \phi^{-1}(a) \,\}$ is independent of $x\in Q_s^{\alpha'_i}$ because of part (b) of the claim, which has been established already. The same is true for $\beta_i$ in place of $\alpha_i'$.
Therefore we can choose a bijection $\sigma_s \colon T \to T$ such that $\sigma_s(e)=e$ and for each $a\in \mathcal A$, we have
\[
| \sigma_s( \{\, t\in T \mid t^{\beta_i}Q_s^{\beta_i} \subset \phi^{-1}(a) \,\} ) \setminus \{\, t\in T  \mid t^{\alpha'_i}Q_s^{\alpha'_i} \subset \phi^{-1}(a) \,\} | \leq 6\varepsilon'|T|+1.
\]
The condition $\sigma_s(e)=e$ will be used in the proof of Claim \ref{weakapp}. We set
\[
T_s = \{\, t \in T \mid t^{\beta_i}Q_s^{\beta_i}\subset \phi^{-1}(a)\ \text{and}\ \sigma_s(t)^{\alpha'_i}Q_s^{\alpha'_i}\subset \phi^{-1}(a)\ {\rm for\ some\ }a\in \mathcal A \,\}.
\]
Then $|T \setminus T_s| \leq (6 \varepsilon' |T|+1)|\mathcal A| < 7\varepsilon'|\mathcal A||T|$ since $|T| > 1/\varepsilon'$, and condition (\ref{equiv1}) holds for any $t \in T_s$ and any $a\in \mathcal A$.
\end{proof}

For an action $\delta \in A(G,X,\mu)$, let $[\delta]_\mathit{SO}$ denote the set of all actions $\eta \in A(G,X,\mu)$ that have the same orbits as $\delta$, i.e., satisfy $G^\delta x = G^\eta x$ for $\mu$-almost every $x \in X$.

\begin{claim}\label{weakapp}
There exists $S_i \in \Aut(X,\mu)$ such that the action $\alpha''_i \coloneqq S_i \alpha'_i S_i^{-1}$ belongs to $[\alpha'_i]_\mathit{SO}$ and for any $g \in F_i$ and any $j\in \{ 1,\ldots,n\}$, we have $\mu(g^{\alpha''_i} A_j \bigtriangleup g^{\beta_i} A_j) < \varepsilon$.
\end{claim}

\begin{proof}

We define $S_i \in \Aut(X,\mu)$ by $S_i = \sigma_s(t)^{\alpha'_i}(t^{-1})^{\alpha'_i}$ on $t^{\alpha'_i}Q_s^{\alpha'_i}$ for $s\in \{ 1,\ldots,r\}$ and $t \in T$ and by $S_i = \text{id}$ on $X \setminus TB^{\alpha'_i}$. We set $\alpha''_i=S_i \alpha'_i S^{-1}_i$. Then $\alpha''_i \in [\alpha'_i]_\mathit{SO}$, and $S_iQ_s^{\alpha_i'}=Q_s^{\alpha_i'}$ since $e\in T$ and $\sigma_s(e)=e$.
By condition (\ref{equiv1}), for any $s\in \{ 1,\ldots,r\}$, any $t \in T_s$ and any $a\in \mathcal A$, we have
\begin{equation}\label{equiv2}
t^{\alpha''_i}Q_s^{\alpha'_i} \subset \phi^{-1}(a) \Leftrightarrow t^{\beta_i} Q_s^{\beta_i} \subset \phi^{-1}(a)
\end{equation}
since $S_i t^{\alpha'_i} S^{-1}_i Q_s^{\alpha'_i} = \sigma_s(t)^{\alpha'_i} Q_s^{\alpha'_i}$.

Fix $g \in F_i$ and $j \in \{1,\ldots,n\}$ arbitrarily.
We set
\[L_0 = X \setminus \bigcup_{t\in T} t^{\alpha''_i}B^{\alpha'_i},\quad L_1 = \bigcup_{t\in T \setminus gT} t^{\alpha''_i}B^{\alpha'_i}\quad \text{and} \quad L_2 = \bigcup_{s=1}^{r}\bigcup_{t \in T \setminus (T_s \cap gT_s)} t^{\alpha''_i}Q_s^{\alpha'_i},\]
We estimate the measures of $L_0$, $L_1$ and $L_2$. We have $\mu(L_0) < 8\varepsilon'$ by condition (2) in Claim \ref{erg+Rohlin}.
Since the set $T$ is $(F_i,\varepsilon')$-invariant and $g \in F_i$, we have $|T \setminus gT| < \varepsilon' |T|$ and hence $\mu(L_1) < \varepsilon'|T|\mu(B^{\alpha'_i}) \leq \varepsilon'$.
Since $|T \setminus gT_s| \leq |T \setminus gT| + |gT \setminus gT_s| = |T \setminus gT| + |T \setminus T_s|$, we have
\[
\mu(L_2) \leq \sum_{s=1}^{r}(|T \setminus gT| + 2|T \setminus T_s|)\mu(Q_s^{\alpha'_i}) < (\varepsilon'|T| + 14\varepsilon'|\mathcal A||T|)\mu(B^{\alpha'_i}) \leq 15|\mathcal A|\varepsilon'.
\]

We claim that $\mu ( (g^{\alpha''_i} A_j \bigtriangleup g^{\beta_i} A_j) \setminus ( L_0 \cup L_1 \cup L_2 ) ) = 0$.
Suppose otherwise, i.e., some $s \in \{1,\ldots,r\}$ and $t \in gT \cap (T_s \cap gT_s)$ satisfy $\mu(t^{\alpha''_i}Q_s^{\alpha'_i} \cap (g^{\alpha''_i} A_j \bigtriangleup g^{\beta_i} A_j)) >0$. By condition (\ref{equiv2}), there exists $a \in \mathcal A$ such that
\begin{equation}\label{cond1}
t^{\alpha''_i}Q_s^{\alpha'_i} \subset \phi^{-1}(a) \quad \text{and} \quad t^{\beta_i}Q_s^{\beta_i} \subset \phi^{-1}(a),
\end{equation}
and there exists $a'\in \mathcal{A}$ such that
\begin{equation}\label{cond2}
{(g^{-1}t)}^{\alpha''_i}Q_s^{\alpha'_i} \subset \phi^{-1}(a') \quad \text{and} \quad {(g^{-1}t)}^{\beta_i}Q_s^{\beta_i} \subset \phi^{-1}(a').
\end{equation}
If $\mu(t^{\alpha''_i}Q_s^{\alpha'_i} \cap (g^{\alpha''_i} A_j \setminus g^{\beta_i} A_j)) >0$, then by condition (\ref{cond2}), we have $\phi^{-1}(a') \subset A_j$ and hence $t^{\beta_i} Q_s^{\beta_i} \subset g^{\beta_i} A_j$. Combining this with condition (\ref{cond1}), we have $\phi^{-1}(a) \subset g^{\beta_i}A_j$ and hence $t^{\alpha''_i}Q_s^{\alpha'_i} \subset g^{\beta_i}A_j$, which contradicts the assumption $\mu(t^{\alpha''_i}Q_s^{\alpha'_i} \cap (g^{\alpha''_i} A_j \setminus g^{\beta_i} A_j)) >0$.
On the other hand, if $\mu(t^{\alpha''_i}Q_s^{\alpha'_i} \cap (g^{\beta_i} A_j \setminus g^{\alpha''_i} A_j)) >0$, then by condition (\ref{cond1}), we have $\phi^{-1}(a) \subset g^{\beta_i}A_j$ and hence $(g^{-1}t)^{\beta_i}Q_s^{\beta_i} \subset A_j$. Combining this with condition (\ref{cond2}), we have $\phi^{-1}(a') \subset A_j$ and hence $t^{\alpha''_i}Q_s^{\alpha'_i} \subset g^{\alpha''_i}A_j$, which contradicts the assumption $\mu(t^{\alpha''_i}Q_s^{\alpha'_i} \cap (g^{\beta_i} A_j \setminus g^{\alpha''_i} A_j)) >0$.
Therefore we have  $\mu(t^{\alpha''_i}Q_s^{\alpha'_i} \cap (g^{\alpha''_i} A_j \bigtriangleup g^{\beta_i} A_j)) =0$ for any $s \in \{ 1,\ldots,r\}$ and any $t \in gT \cap (T_s \cap gT_s)$. As a result, we have $\mu ( (g^{\alpha''_i} A_j \bigtriangleup g^{\beta_i} A_j) \setminus ( L_0 \cup L_1 \cup L_2 ) ) = 0$ and
\[
\mu ( g^{\alpha''_i} A_j \bigtriangleup g^{\beta_i} A_j) \leq \mu(L_0) + \mu(L_1) + \mu(L_2) < 8\varepsilon' + \varepsilon' + 15|\mathcal A|\varepsilon' \leq 24|\mathcal A|\varepsilon' =\varepsilon.
\]
This holds for any $g \in F_i$ and any $j\in \{ 1,\ldots,n\}$.
\end{proof}

Let $\gamma \in A(G, X, \mu)$ be the action defined by $\gamma |G_i = \alpha''_i$ for each $i\in \{ 1, \ldots, k\}$. Then we have $\gamma \in [\alpha]_\mathit{OE}$ since $\alpha' \in [\alpha]_\mathit{MC}$ and $\alpha''_i \in [\alpha'_i]_\mathit{SO}$ for every $i\in \{ 1,\ldots,k\}$.
Claim \ref{weakapp} shows that $\gamma$ is the desired action.
This completes the proof of Theorem \ref{main}.
\end{proof}

\begin{proof}[Proof of Theorem \ref{thm-intro}]
It remains to show the theorem for a free product of (countably) infinitely many, infinite abelian groups. We show this, following \ci{Corollary 2.2}{bo}. 
Let $G=G_1 \ast G_2 \ast \cdots$ be the free product of infinite abelian groups $G_i$ with $i\in \N$.
Suppose that $\alpha \in \FR(G,X,\mu)$ and $\beta \in A(G,X,\mu)$. Let $F \subset G$ be a finite subset, let $\varepsilon >0$, and let $A_1,\ldots,A_n \subset X$ be Borel subsets. We set $H_k = G_1 \ast \cdots \ast G_k$ for $k \in \N$. Pick $k \in \N$ such that $F \subset H_k$. By Theorem \ref{main}, there exists an action $\alpha' \in A(H_k,X,\mu)$ such that $\alpha' \in [\alpha|H_k]_\mathit{OE}$ and $\mu(g^{\alpha'} A_j \bigtriangleup g^{\beta} A_j ) < \varepsilon$ for any $g \in F$ and $j\in \{ 1,\ldots,n\}$. Since $\alpha' \in [\alpha|H_k]_\mathit{OE}$, there exists $R \in \Aut(X,\mu)$ such that $R^{-1} \alpha' R \in [\alpha|H_k]_\mathit{SO}$. We define an action $\gamma \in A(G,X,\mu)$ by $\gamma|H_k = \alpha'$ and $g^{\gamma} = Rg^{\alpha}R^{-1}$ for $g \in G_m$ with $m \geq k+1$. Then $\gamma$ is orbit equivalent to $\alpha$ via $R$ and satisfies $\mu(g^\gamma A_j\bigtriangleup g^\beta A_j)<\varepsilon$ for any $g\in F$ and $j\in \{ 1,\ldots, n\}$.
\end{proof}

%%%%%%%%%%%%%%%%%%%%%%%%%%%%%%%%%%%%%%%%%%%%%%

\end{document}